\documentclass[a4paper]{article}
\usepackage{amssymb}
\usepackage{amsmath}
\usepackage[hyphens]{url} 
\usepackage[utf8]{inputenc}
\usepackage{mathtools}  
\usepackage{amsthm}
\usepackage{color}
\usepackage{nicefrac}
\usepackage{tcolorbox}
\usepackage{xcolor}

\newcommand{\la}{\langle}
\newcommand{\ra}{\rangle}
\newcommand{\R}{\mathbb{R}}

\newcommand{\Prob}{\mathbb{P}}

\newcommand{\p}{\tilde{p}}
\newcommand{\z}{\tilde{z}}
\newcommand{\x}{\tilde{x}}

\newcommand*\samethanks[1][\value{footnote}]{\footnotemark[#1]}

 \newtheorem{thm}{Theorem}[section]
\newtheorem{remark}{Remark}

 \newtheorem{prop}{Proposition}[section]
 \newtheorem{lem}{Lemma}[section]
 
 \newtheorem{ass}{Assumption}

 \newtheorem{dfn}{Definition}

\title{Network manipulation algorithm based on inexact alternating minimization}

\author{David Müller \thanks{Department of Mathematics, Chemnitz University of Technology,
Reichenhainer Str. 41, 09126
Chemnitz, Germany; e-mail: david.mueller@mathematik.tu-chemnitz.de (corresponding author), vladimir.shikhman@mathematik.tu-chemnitz.de}  \and Vladimir Shikhman\samethanks[1]  } 
\providecommand{\keywords}[1]{\textbf{\textit{Keywords---}} #1}

\begin{document}

\maketitle
\vspace{18mm} \setcounter{page}{1}

\begin{abstract}	
In this paper, we present a network manipulation algorithm based on an alternating minimization scheme from \cite{nesterov2020soft}. In our context, the latter mimics the natural behavior of agents and organizations operating on a network. 
By selecting starting distributions, the organizations determine the short-term dynamics of the network. While choosing an organization in accordance with their manipulation goals, agents are prone to errors given by discrete choice probabilities. We extend the analysis of our algorithm to the inexact case, where the corresponding subproblems can only be solved with numerical inaccuracies. The parameters reflecting the imperfect behavior of agents and the credibility of organizations, as well as the condition number of the network transition matrix have a significant impact on the convergence of our algorithm. 
Namely, they do not only improve the rate of convergence, but also reduce the accumulated errors.
\end{abstract}

\section{Introduction}
Networks naturally occur in many areas such as economics, computer science, chemistry or biology. A common way to model scenarios within networks is to use Markov chains. For a finite state space, a transition matrix describes the structure of changes on the chain's states \cite{gagniuc2017markov}. Usually, the iterative process of repeated transitions over the states provides a stationary distribution. 
However, if the considered time horizon is short, a question arises on how to efficiently manipulate the distribution of information within the network. As an obvious choice for manipulation, each agent may start with an initial distribution and spread the information by communicating with the neighbors. But, since the own manipulation power is often limited, it is quite reasonable for an agent to engage intermediary organizations instead. This could be due to the restricted access to the parts of the network, revealing the manipulation interests too apparently, or due to the lack of knowledge about the network structure. Possible examples include:
\begin{itemize}
	\item[] {\bf Influencers.} Companies, who want to credibly advertise their products or services via social media channels, pay influencers on social media platforms. They act as a part of the network and spread information on the products or services.
	\item[] {\bf Search engines.} Web site owners try to increase the visibility of their web sites. In order to find proper content, most of the Internet users enter a query into a web search engine. Thus, these query results    strongly influence the short term behavior of the users.
	\item[] {\bf Conspiracy theory.} Agents try to spread false information for different interests. 
	In order to spread fake news, agents have to rely on different distribution channels, such as groups in social media networks or Internet blogs.  
\end{itemize} 
Note that two of our examples are related to the manipulation of information in a social network. There is a growing literature concerning this topic, see  \cite{acemoglu2011opinion} for an overview.  Most of the models describe the update of opinions or beliefs, see e.\,g. \cite{degroot1974reaching}, where the latter is done according to a convex combination of other network members' opinions. Applying traditional techniques from the analysis of Markov chains, the formation of a consensus is examined. In these approaches, manipulation is modeled by modifying the transition matrix, e.\,g. introducing randomness \cite{acemoglu2011opinion}. In \cite{forster2014trust},   manipulation in a model of opinion formation is studied. There, the weights of the transition matrices can be changed by agents, while all starting distributions are fixed. Our model differs from those in the existing literature, since we examine how the information regarding a topic is distributed among network participants through intermediaries. Loosely speaking, we analyze who knows how much and how this information state can be efficiently manipulated by engaging intermediaries. As mentioned in \cite{acemoglu2011opinion}, one central component of opinion formation is how agents update their prior beliefs based on new information. In this paper, we also contribute to the opinion formation, because we investigate a way to manipulate the acquirement of information by employing the network of information sources. Our goal is not to learn the complete structure of the  network, for which  usually hidden Markov models are applied, see e.\,g. \cite{yang1997human}. Instead, organizations should be able to select a starting distribution aiming to arrive at a certain information state after a number of iterations. Agents are choosing among the intermediare organizations to boost manipulation. This is at the core of our network manipulation algorithm. Note that a similar problem has been analyzed  in \cite{lindqvist1977fast}, where the author applies decision-theoretic techniques to observe a state at a time and obtain information about the initial state. 

Let us comment on the mathematics behind the proposed network manipulation algorithm. It is motivated by \cite{nesterov2020soft} where a new technique for soft clustering is introduced. For this, voters and political parties alternately solve their subproblems, yielding an alternating minimization scheme. The behavior of voters turns out to be in accordance with the well known multinomial logit model from discrete choice theory. Namely, the voters choose rationally among the parties, but are prone to random errors, see e.\,g. \cite{palma}. The parties update their political positions depending on how many voters they attract. Overall, the resulting soft clustering is given in terms of probabilities from the multinomial logit model. In this paper, we generalize the idea suggested in \cite{nesterov2020soft} to a broader class of discrete choice probabilities. This is done by presenting a network manipulation model based on alternating steps performed by agents and organizations. Agents try to manipulate a network by choosing intermediary organizations for helping in that. In order to select among the organizations, agents observe which of them better manipulate the network in comparison to agents' goals. While doing so, agents are prone to random errors, which lead to choice probabilities following certain discrete choice models examined in our previous paper \cite{prox}. Altogether, we show how the alternating minimization scheme of \cite{nesterov2020soft} can be applied for network manipulation. Additionally, we present an inexact version of the alternating minimization scheme introduced in \cite{nesterov2020soft}. Inexactness is due to the fact that the subproblems of agents and/or organizations may not be solved exactly and may suffer from numerical inaccuracies. Overall, we conclude that the agents' imperfect behavior and organizations' conservatism in profit maximization reduce the accumulated errors.

{\bf Notation.}  
In this paper, we mainly focus on subspaces of $\R^n$ and $\R^{m \times n}$, where the former is the space of $n$-dimensional column vectors $$x =
\left(x^{(1)}, \dots , x^{(n)}\right)^T,$$
and the latter denotes the linear space of $(m \times n)$-matrices.
We denote by $e_j \in \R^n$ the $j$-th coordinate vector of $\R^n$ and write  $e$ for the vector of an appropriate dimension whose components are equal to one. By $\R^n_+$ we denote the set of all vectors with nonnegative components and notation $\Delta_n$ is used for the standard simplex
\[
\Delta_n = \left\{ x \in \R^n_+ : \sum_{i=1}^{n} x^{(i)} = 1\right\}.
\] We use the norms for $x \in \R^n$: 
\[
\|x\|_2 = \left[\sum_{i=1}^{n}  x^{(i)2}\right]^\frac{1}{2}, \quad \|x\|_1 = \sum_{i=1}^{n} \left| x^{(i)}\right|, \quad
\|x\|_\infty = \max_{i=1, \ldots, n} \left| x^{(i)}\right|.
\]
 For $x, s \in \R^n$ we use the standard scalar product:
\[
\la x, s \ra = \sum\limits_{i=1}^n x^{(i)} s^{(i)}.
\]
For matrices $A,B \in \R^{m\times n}$ the inner product is defined via the trace:
\[
\la A,B\ra = \mbox{Tr}(A^T\cdot B).
\]


A function $F:Q \to \R$ is called $\beta$-strongly convex on a convex and closed set $Q \subset \R^n$ w.r.t. a norm $\|\cdot\|$ if for all $x,y \in Q$ and $\alpha \in [0,1]$ it holds:
	\[
	F(\alpha x +(1-\alpha)y) \leq \alpha F(x) + (1-\alpha) F(y) - \alpha (1-\alpha) \cdot \frac{\beta}{2} \|x-y\|^2.
	\]
The positive constant $\beta$ is called the  convexity parameter of $F$. If $\beta=0$, we call $F$ convex.
A function $\pi$ is $\beta$-strongly concave if  $-\pi$ is  $\beta$-strongly convex.
For a convex function $F:Q \to \R$ the set $\partial F(x)$ represents its subdifferential at $x \in Q$, i.e.
\[
\partial F(x) = \left\{ g \in \R^n : F(y) \ge F(x) + \la g, y - x\ra \, \mbox{for all} \, y \in Q \right\}.
\] 
Its convex conjugate is
\[F^*(s) = \underset{x \in \mathbb{R}^n}{\sup}\left\{\left\langle x,s \right\rangle - F(x)\right\},\] 
where $s \in R^n$ is a vector of dual variables.
We denote by $\nabla F(x)$ the gradient of a differentiable function $F$ at $x$.

\section{Manipulation Model}\label{ch:model}
Let us introduce our model in order to later construct a manipulation algorithm based on  interaction  within a network. 

\subsection{Interaction network}
A central aspect in our model is a network with $n$ nodes.  The structure of this network describes how nodes interact among each other, e.\,g. how persons receive and exchange information.  Thereby, a link from node $j$ to $i$ represents a connection. In the context of an information network, the latter would depict that  person $i$ acquires information from person $j$.  We summarize the data in a transition matrix $M = \left(M_{ij}\right)_{i,j =1}^{n}$, where $M_{ij}$ denotes the transition probability of node $j$ to  node $i$. Hence, the following holds:
\[\sum\limits_{i=1}^{n} M_{ij} = 1 \quad \text{for} \quad j=1,\ldots, n.\]  
$M$ is a column stochastic matrix, i.e. $ M\ge 0, \; e^T\cdot M = e^T$. Our model describes the process of information acquirement rather than the formation of opinions as e.\,g. in \cite{forster2014trust}. We are interested in a few periods of interactions, thus, we take the transition matrix as fixed. 
The interaction  causes different states of  the network, based on the connections of  its nodes. The state of a network can be represented as an element of the standard simplex in $\R^n$ dependent on time variable. We call a vector $x(t) \in \Delta_n$ a state of a network at time $t$. Such a state reflects the  value  each node possesses after an interaction with other nodes. This could be for example the amount of information a person possesses in relation to the others or the market share of a company. 

The dynamics of interaction can be described  by an  iterative process.  The latter starts with a vector $x(0) \in \Delta_n$ and then the nodes interact repeatedly with each other. Thus, the iterative process is given by   
\begin{equation}\label{eq:iter.proc}
	x(t) = M\cdot x(t-1)= \ldots = M^t\cdot x(0).
\end{equation}
Obviously, all $x(t)$'s generated according to this process are elements of $\Delta_n$. Our idea is closely related to  the concept of network rankings, such as the famous {\it PageRank} \cite{page1999pagerank}. However, we focus on a limited, mostly small number, of interaction periods. Within an information network, persons would typically exchange information for a few periods, before they make a decision. This short term behavior endows the starting vector $x(0) \in \Delta_n$ with importance. For the sake of brevity, we  drop the time index by writing $x=x(0)$.

\subsection{Agents}\label{sec:agents}
Let us assume that agents want to manipulate the resulting state of a network in favor of their own interests. 
Though they aspire certain network states, agents face some challenges by trying to manipulate a network. Often, they do not have knowledge of the network structure. Additionally, there are many situations, where the agents can't participate in the network because they cannot connect to a node without revealing their intentions, e.\,g. companies cannot credibly advertise their products by themselves. There might also be networks, where an agent could interact in, but is restricted to start with a fixed vector. In particular, if the information is just spread uniformly. Instead, agents could instruct organizations to manipulate the interaction in order to  reach an aspired state of the network. 
The organizations often have more information or at least experience about the structure of a network. In fact, they could even operate it.
The agents choose among $K$ organizations, where each organization provides an observable utility $u^{(k)}$. 
This discrete choice behaviour we describe by means of the so-called additive random utility models. The additive decomposition of utility goes back to psychological experiments accomplished in the 1920's \cite{thurstone}. A formal description of this framework has been first introduced in an economic context \cite{gev}, where rational decision-makers choose from a finite set of mutually exclusive alternatives $\{1, \ldots, K\}. $ Although the decision rule follows a rational behavior, agents are prone to random errors. The latter describe decision-affecting features which cannot be observed. Each alternative $k = 1,\ldots, K $ provides the utility \[u^{(k)} + \epsilon^{(k)}, \] where $ u^{(k)} \in \mathbb{R} $ is the deterministic utility part of the $k$-th alternative and $\epsilon^{(k)} $ is its stochastic error.  We use the following notation for the vectors of deterministic utilities and of random utilities, respectively:
\[
u = \left(u^{(1)}, \ldots, u^{(K)}\right)^T, \quad \epsilon = \left(\epsilon^{(1)}, \ldots, \epsilon^{(K)}\right)^T. 
\] 
The probabilistic framework yields choice probabilities for each alternative:
\begin{equation}\label{eq:prob}
	p^{(k)} = \Prob \left( u^{(k)} + \epsilon^{(k)} = \max_{1 \leq m \leq K} u^{(m)} + \epsilon^{(m)}\right), \quad k=1, \ldots, K. 
\end{equation}
As the consumers behave rationally, their surplus is given by the expected maximum utility of their decision:
\begin{equation}\label{eq:cons_surplus}
E(u) = \mathbb{E}_\epsilon  \left(\max_{1 \leq k \leq K} u^{(k)} + \epsilon^{(k)}\right).   
\end{equation}
It is well known that the surplus function is convex \cite{palma}.
Additionally we make a standard assumption concerning the distribution of random errors. 
\begin{ass}[]\label{ass:joint density}
	The random vector $\epsilon $ follows a joint distribution with zero mean that is absolutely continuous with respect to the Lebesgue measure and fully supported on $\mathbb{R}^K$.
\end{ass} 

\begin{ass}\label{ass:agents_behavior}	 The differences $\epsilon^{(k)} - \epsilon^{(m)}$  of random errors have finite modes for all $k,m=1,\ldots,K$.
\end{ass}

Assumption \ref{ass:joint density} guarantees that no ties in occur in \eqref{eq:cons_surplus}, which provides differentiability of the surplus function.   Further, the gradient of $E$ corresponds to the vector of choice probabilities, which is known as the Williams-Daly-Zachary theorem \cite{gev}, i.e.
\begin{equation}\label{eq:Daly}
\frac{\partial E}{\partial u^{(k)}} = p^{(k)}, \quad k=1, \ldots, K.
\end{equation}
Hence, each component of the gradient of $E$  yields  the probability that alternative $k$ provides the maximum utility among all alternatives.

Another equivalent representation of choice probabilities can be obtained by means of the convex conjugate of the surplus function ${E^*: \R^K \to \R \; \cup \; \left\{\infty\right\}}$:
\[
E^*(p) = \sup_{u \in \R^K} \left\{ \la p, u \ra - E(u)\right\},
\]
where $p = \left(p^{(1)}, \ldots, p^{(K)}\right)^T$ is the vector of dual variables.
In view of conjugate duality and , the vector of choice probabilities can be derived from an optimization problem of rational inattention, see e.\,g. \cite{shum} and \cite{prox}. Indeed, it has been shown that under Assumption \ref{ass:agents_behavior}  the  vector of choice probabilities $\Prob$ is the unique solution of 
\begin{equation}\label{eq:conj_dual}
\underset{p  \in \Delta_K}{\max} \left\{\la u,p \ra - E^*(p)\right\}.
\end{equation}

Now, we assume that there are $N$ agents trying to manipulate the network.
Each agent $i$ has an aspired state of network which we denote by $v_i \in \Delta_n$. 
  In order to reach the aspired state, agents can choose  among $K$ organizations.  The $k$-th organization is able to manipulate the interaction dynamics  in the network, which yields at time $t$ a state of a network $x_k(t) \in \Delta_n, \; k=1, \ldots, K$. 
In general agents prefer  organizations which provide a network state in line with the states they desire such as an aspired market shares distribution or state of information.  In order to assess the outcome of a manipulation, any  agent $i$ has to compare $K$  distances, i.e.
\[
\|v_i - x_k(t)\|_2 , \quad k=1, \ldots, K,
\]
respectively
\begin{equation}\label{eq:dist}
\|v_i - M^t\cdot x_k\|_2 , \quad k=1, \ldots, K.
\end{equation}
Note that \eqref{eq:dist} provides a way for agent $i$ to observe the utility of choosing the $k$-th organization. 
The network state at time $t$ is observable, so any agent is able to check, if an organization has manipulated the network satisfactorily. Let us put all the states   in a matrix, which yields a way to summarize all the states of a network at time $t$ in one variable, i.e.
\[
X(t)= \left(x_1(t), \ldots, x_K(t)\right) \in \Delta_n^K.
\] 
The matrix above can also be expressed in terms of the starting vectors, by defining 
\[
X = \left(x_1, \ldots, x_K\right) \in \Delta_n^K,
\]
which enables us to write 
\begin{equation}\label{eq:matrix}
	X(t) = M^t \cdot X.
\end{equation}

We define a vector valued function  $g_i: \Delta_n^K \rightarrow \mathbb{R}_+^K$ for any agent $i$, which stores  all these distances of the $i$-th agent and, hence, depends on a matrix $X$ as input variable:
\[
g_i(X) = \left(\|v_i - M^t\cdot x_1\|_2, \ldots, \|v_i - M^t\cdot x_K\|_2\right)^T,\quad i=1,\ldots, N. 
\]
We write in matrix form:
\[
   G(X) = \left(g_1(X), \ldots, g_N(X)\right) \in \R^{K \times N}.
\]
In view of additive random utility models, $g_i(\cdot)$ provides a way to characterize the observable utility $u_i$ by setting 
\[
 u_i = -g_i(X), \quad i=1,\ldots,N.
\]
Hence, the  vector of the $i$-th agent choice probabilities has entries
\begin{equation}\label{eq:Ag.ChoiceProb}
	p_i^{(k)}(X) = \Prob \left( -g_i^{(k)}(X) + \epsilon_i^{(k)} = \max_{1 \leq m \leq K} -g_i^{(m)}(X) + \epsilon_i^{(m)}\right), \quad k=1, \ldots, K. 
\end{equation}
Equivalently, $p_i(X)$ solves the following rational inattention problem:
\[
  \min_{p  \in \Delta_K} \la g_i(X),p \ra + E_i^*(p).
\]
Let us stack the choice probabilities of all the agents into a matrix and call the latter choice matrix:
\begin{equation}\label{eq:ChoiceMatr}
P(X) = \left(p_1(X), \ldots, p_N(X)\right) \in \Delta_K^N.
\end{equation}
Similarly to the choice matrix, we write $ P \in \Delta_K^N$
for any matrix of  probability vectors, 
i.e. $P = \left(p_1, \ldots, p_N\right)$ with $p_i \in \Delta_K$, $i=1, \ldots, N$.

\subsection{Organizations}
Let us describe the behavior of advertising organizations. Their goal is to attract agents as clients by providing them with additional manipulation power. This is done by choosing an appropriate starting distribution, thus, the communication process is initialized by organizations. By strategic decisions, such as substantial alignment, design, product placements or personal relations, or by direct decisions, such as ranking of a website as result of a certain query or advertising products directly on a marketplace, the organizations determine these vectors, which reflect a network state before interaction  starts. 
In order to attract the $i$-th agent with aspired state $v_i$, the $k$-th organization selects a starting distribution $x_k \in \Delta_n$ such that $\left\|v_i - M^t\cdot x_k \right\|_2$ becomes small. The organization's goal is to acquire as many agents as possible by simultaneously satisfying the corresponding aspired states. However, the agents are not necessarily equally important for the organization. Instead, the latter primary wants to  please agents, who already prefer the organization compared to other competitors. Let us state these considerations in a formal way.  An organization $k$ observes to which extent the agents choose it, i.e. quantified by choice probabilities $p_i^{(k)}$, $i=1, \ldots, N$. Thus, the $k$-th organization measures its performance by the following objective:
\begin{equation}\label{eq:org.WeightedSum}
	 \sum_{i=1}^{N} p_i^{(k)}\cdot \|v_i-M^t\cdot x_k\|_2.
\end{equation}
Yet, an organization's choice of the manipulation distribution not only depends on the agents' aspired states, but also on its own objectives. This reflects, that an organization might also aspire a certain state of the network in order to gain profits from the network participants. Therefore, we introduce a payoff function for organization $k$, which depends on the latter's caused state of manipulation:
\begin{equation}\label{eq:Payoff.Fct}
\pi_k\left(M^t\cdot x_k\right).
\end{equation}
Let us illustrate by examples how a network state could affect the payoff of  the $k$-th organization. Groups in social media platforms might avoid sharing  information with persons who have contrary opinions, such that no arguments against their theories or fake news are communicated. Prohibiting or restricting persons' access to information might be a worthwhile purpose in an information network. Particularly, this is interesting in situations, where direct manipulation of opinions is difficult. Since the authors in \cite{acemoglu2011opinion} mention the source of information as a key component of opinion formation, the manipulation of the information acquirement process contributes to the tampering of opinion formation. A social media influencer might loose credibility of her followers, if the latter find out about an unacceptable advertise. We state an assumption concerning the payoff functions.
\begin{ass}\label{ass:Payoff.Fct}
	The payoff function $\pi_k$ is $\tau_k$-strongly concave w.r.t. the norm $\|\cdot\|_2$ for all $k=1, \ldots, K$.
\end{ass} 
Altogether, the objective function of the $k$-th organization incorporates the both goals:
\begin{equation}\label{eq:org.ObjFunct}
	\sum_{i=1}^{N} p_i^{(k)}\cdot \|v_i-M^t\cdot x_k\|_2 
	- \frac{1}{\eta_k} \cdot \pi_k\left(M^t\cdot x_k\right),
\end{equation}
where $\eta_k > 0$ is a regularization parameter, which shows the importance of payoffs generated by the network. Note that thereby, small values of $\eta_k$ indicate a more restrictive behavior of the $k$-th organization, meaning the latter rather focuses on its own interests than to freely adjust the manipulation distribution according to the agents' aspired states.  For a given choice matrix $P \in \Delta_K^N$, network $M$ and   time $t$, the $k$-th organization chooses its optimal starting distribution $x_k \in \Delta_n$ by solving 
	\begin{equation}\label{eq:OrgMin}
	 \min_{x_k \in \Delta_n} \sum_{i=1}^{N} p_i^{(k)}\cdot \|v_i-M^t\cdot x_k\|_2 - \frac{1}{\eta_k} \cdot \pi_k\left(M^t\cdot x_k\right).
	\end{equation} 
For now, we assume that the optimization problems given in \eqref{eq:OrgMin} have unique solutions  for any choice matrix $P$, which we denote by $x_k(P)$, $k=1, \ldots, K$. 
We keep these optimal manipulation values in a matrix 
\begin{equation}\label{eq:X.Minimizer}
X(P) =\left(x_1(P), \ldots, x_K(P)\right),
\end{equation}
and call the latter manipulation matrix. 

\subsection{Network manipulation algorithm}
\label{subsec:algorithm}
In the preceding section we described the behavior of agents and organizations when facing the challenge to manipulate a network in favor of agents' desires. The key aspect is that their behavior summarized in \eqref{eq:Ag.ChoiceProb} and \eqref{eq:OrgMin} suggests   an alternating interaction between both groups. Organizations enter the market and offering their manipulation distributions. Then, agents observe  how satisfactory organizations  would manipulate the network state in view of the agents aspired states (e.\,g. by comparing past results caused by an organization). Based on these observations, agents make their decisions, i.e. they choose organizations with probability according to equation \eqref{eq:Ag.ChoiceProb}. The choice probabilities provide feedback to the organizations, which then in turn adjust their starting distributions following the behavior given in  \eqref{eq:OrgMin}. By using previous notation, we have the following dynamics: 
	\[
	P_{\ell+1} = P(X_\ell), \quad 
		X_{\ell+1} = X(P_{\ell+1}),
	\]
where $X_0$ is any feasible starting variable e.\,g. $X_0 = \frac{1}{n}\cdot ee^T$.

In what follows, we provide an equivalent description of this network manipulation algorithm in order to better study its convergence properties. For that, we define a  potential function which incorporates the behavior of all agents and organizations: 
\begin{equation}\label{eq:Potential}
\Phi(X,P) = \sum_{i=1}^{N} E_i^*(p_i) + \sum_{k=1}^{K} \left(\sum_{i=1}^{N} p_i^{(k)}\cdot \|v_i-M^t\cdot x_k\|_2- \frac{1}{\eta_k} \cdot \pi_k\left(M^t\cdot x_k\right)\right).
\end{equation}
Therefore, the choice matrix solves the following minimization problem:
 \begin{equation}\label{eq:PartiMin.P}
 P(X) = \arg\min_{P \in \Delta_K^N} \Phi(X,P).
 \end{equation}
 Analogously, we have for a manipulation matrix:
 \begin{equation}\label{eq:PartMin.X}
 X(P) = \arg\min_{X \in  \Delta_n^K} \Phi(X,P),
 \end{equation}
  which means that the network manipulation algorithm can be viewed as an alternating minimization scheme. 

From the viewpoint of computational economics, it seems reasonable to assume that agents and organizations are not able to solve their corresponding optimization problems exactly. Rather than that, the solutions can be obtained up to small errors. This can be for example due to observation errors of the input parameters given by choice and/or manipulation matrices. Another reason could be that exact optimization is time-exhaustive or too costly. In order to incorporate this faulty behavior into our manipulation algorithm, we assume that just inexact minimization in (\ref{eq:PartiMin.P}) and (\ref{eq:PartMin.X}) is possible.
More precisely, $\delta_1$-inexact solutions for (\ref{eq:PartiMin.P}) and $\delta_2$-inexact solutions for (\ref{eq:PartMin.X}) are available, see Section \ref{sec:inex.alt} for details.
Thus, we are ready to state a more general network manipulation algorithm, based on an inexact alternating minimization scheme:
 \begin{tcolorbox}
 		Initialize $\tilde{X}_0 \in \Delta_n^K$. For $\ell=0, 1, 2, \ldots$ update:
 		\begin{itemize}
 			\item[] $\displaystyle \tilde{P}_{\ell+1} = \arg\min^{\delta_1}_{P \in Q_2} \Phi\left(\tilde{X}_\ell,P\right)$,
 			\item[] $\displaystyle \tilde{X}_{\ell+1} =\arg\min^{\delta_2}_{X \in Q_1} \Phi\left(X,\tilde{P}_{\ell + 1}\right)$.
 		\end{itemize}
 \end{tcolorbox}
The inexact  algorithm raises the questions, if the corresponding alternating behavior converges to a stable equilibrium. Do agents and organizations reach a state, where their choices do not change anymore no matter what the starting distributions of the organizations look like? In other words, does a unique minimizer of the  potential function exist and does the algorithm converge to the latter. Moreover, it is interesting to analyze how the faulty behavior in terms of the errors impacts the possible convergence. We shall answer these questions by applying  general results on inexact alternating minimization schemes, which we present in Section \ref{sec:inex.alt}. This is possible since the potential function \eqref{eq:Potential} can be suitably decomposed.
For that, we define:
\begin{equation}\label{eq:dfn.f}
f(X) =  -\sum_{k=1}^{K} \frac{1}{\eta_k} \cdot \pi_k\left(M^t\cdot x_k\right), \quad
h(P) = \sum_{i=1}^{N} E_i^*(p_i).
\end{equation}
Using the standard inner product, the potential function in  \eqref{eq:Potential} can be written as follows:
\begin{equation}\label{eq:Pot.2}
\Phi(X,P) =  f(X) + \la G(X),P \ra + h(P).
\end{equation} 

\section{Inexact alternating minimization}
\label{sec:inex.alt}

In cases, where the analytical solution of an optimization problem cannot be derived, it is necessary to solve the latter numerically. Normally, this numerical solutions are only exact up to a small $\delta$-error.
We review some theoretical aspects of inexact optimization, which we need for convergence analysis. 
Let us consider optimization problems of the form
\begin{equation}\label{eq:opt.prob}
\min_{z \in Q} \Phi(z),
\end{equation}
where $\Phi$ is a strongly convex function and $Q$ a closed and convex set. We denote by $z^*$  the solution of problem \eqref{eq:opt.prob}.
Recall that for a $\beta$-strongly convex function $\Phi$ it holds:
\begin{equation}\label{eq:oc}
\Phi(z) \ge \Phi(z^*) + \frac{\beta}{2}\|z-z^*\|^2 \quad \text{for all} \; z \in Q. 
\end{equation}
For a $\delta$-inexact solution we use the standard definition, see e.\,g. \cite{stonyakin2019gradient}:
\begin{dfn}\label{dfn:delta.sol}
	A point $\z$ is a $\delta$-inexact solution with $\delta \geq 0$, i.e. 
	\[
	\z \in \arg\min^\delta_{z \in Q} \Phi(z),
	\]
	if and only if there exists $g \in \partial \Phi(\z)$ such that 
	$\la g, z^* - \z\ra \ge - \delta$.
\end{dfn}
Due to Definition \ref{dfn:delta.sol}, a point $\z$ provides the minimal  objective function value of (\ref{eq:opt.prob}) up to the error $\delta$. This can be easily seen, because $\Phi$ is convex and therefore it holds for any $z \in Q$:
\[
\Phi(z^*) \ge \Phi(\z) + \la g, z^* - \z\ra \ge \Phi(\z) - \delta,
\]
which is equivalent to 
\begin{equation}
\label{eq:genau}
\Phi(\z) \le \Phi(z^*) + \delta \le \Phi(z) + \delta.
\end{equation} 

In what follows, we shall focus on decision variables $z$, which can be separated into two blocks, i.e. $z= (x,p)$. For those situations, alternating minimization methods can be applied. The block structure enables to minimize the objective function for each block separately, which is, in particular, a valuable property for big data applications. Over years, many convergence results for alternating minimization methods under different assumptions were shown, see e\,.g. \cite{grippof1999globally}, \cite{luo1993error}, and \cite{beck2015convergence}. The authors in \cite{7040315} show, that under assumptions such as convexity for one and strong convexity for the other objective term, the inexact alternating minimization algorithm applied to the primal coincides with the inexact proximal gradient method to the dual problem. Recently, in \cite{nesterov2020soft} an alternating minimization method was used for soft clustering. There, the objective function additionally includes an interaction term linking both blocks of variables. Under certain assumptions, linear convergence was established  provided the problem can be solved exactly in each block. In this chapter, we are interested in an inexact alternating minimization algorithm for objective functions equipped with the structure introduced in \cite{nesterov2020soft}. Let $Q_1, Q_2$ be closed and convex sets in a finite dimensional vector spaces, and the objective function is given by
\begin{equation}\label{eq:obj.fct.nest}
\Phi(x,p) = f(x) + \la G_1(x),G_2(p) \ra + h(p), \quad x \in Q_1, p \in Q_2, 
\end{equation}
where  the functions $f$ and $h$ are  $\sigma_1$- and $\sigma_2$-strongly convex, respectively, and the functions  $G_1$ and $G_2$ are Lipschitz-continuous with moduli $L_1$ and $L_2$, respectively, such that the interaction term $\la G_1(x),G_2(p) \ra$ is convex and closed in $x \in Q_1$ for any fixed $p \in Q_2$ and vice versa. Further, we assume  the following strict inequality to hold
\begin{equation}\label{eq:conv.cond}
L_1^2\cdot L_2^2 < \sigma_1 \cdot \sigma_2,
\end{equation}
under which the function $\Phi$ is shown to be strongly convex on $Q = Q_1\times Q_2$, see \cite{nesterov2020soft}. 
Let the optimal solution of \eqref{eq:opt.prob} be written as $z^* = (x^*, p^*)$. In order to solve \eqref{eq:opt.prob}, an alternating minimization method has been proposed in \cite{nesterov2020soft}. The latter generates sequences $\{x_\ell\}_{\ell \ge 0}$ and $\{p_\ell\}_{\ell \ge 1}$ as follows:
\begin{tcolorbox}
Choose $x_0 \in Q_1$. For $\ell=0, 1, 2, \ldots$ update:
 		\begin{itemize}
 			\item[] $\displaystyle p_{\ell+1} = \arg\min_{p \in Q_2} \Phi(x_\ell,p)=u(x_\ell)$,
 			\item[] $\displaystyle x_{\ell+1} = \arg\min_{x \in Q_1} \Phi(x,p_{\ell+1})= v(p_{\ell+1})$.
 		\end{itemize}
\end{tcolorbox}
Convergence analysis in \cite{nesterov2020soft} is based on fixed point iteration. For that, the operators $T:Q_1 \mapsto Q_1$  and $S:Q_2 \mapsto Q_2$ are defined as follows:
\begin{equation}\label{eq:exact.oper}
T(x) = v(u(x)), \quad
S(p) = v(u(p)).
\end{equation}
This enables  to write the update step of the alternating minimization scheme:
\begin{align*}
&x_{\ell+1} = T(x_\ell), \quad  p_{\ell+1} = S(p_\ell).
\end{align*}
Under condition \eqref{eq:conv.cond}, $T(\cdot)$ and $S(\cdot)$ are contraction mappings. Thus, the linear convergence of the generated sequences to the minimizer $\left(x^*,p^*\right)$ of $\Phi$ could be shown in \cite{nesterov2020soft}: 
\[
  \|x_{\ell+1} - x^*\| \leq \lambda^{\ell+1} \|x(0)-x^*\|, \quad
  \|p_{\ell+1} - p^*\| \leq \lambda^{\ell} \|p(1)-p^*\|,
\]
where
\[
\lambda = \frac{L_1^2\cdot L_2^2}{\sigma_1 \cdot \sigma_2} <1.
\]


We analyze an inexact version of the alternating minimization method applied to objective functions in \eqref{eq:obj.fct.nest}, when subproblems are solved inexactly in the sense of Definition \ref{dfn:delta.sol}. For that, let us adapt the algorithm in the following way:
\begin{tcolorbox}
Choose $x_0 \in Q_1$. For $\ell=0, 1, 2, \ldots$ update:
 		\begin{itemize}
 			\item[] $\displaystyle \p_{\ell+1} = \arg\min^{\delta_1}_{p \in Q_2} \Phi(\x_\ell,p) =u^{\delta_1}(\x_\ell)$,
 			\item[] $\displaystyle \x_{\ell+1} = \arg\min^{\delta_2}_{x \in Q_1} \Phi(x,\p_{\ell+1}) =v^{\delta_2}(\p_{\ell+1})$.
 		\end{itemize}
\end{tcolorbox}
 We allow different  accuracy for the above subproblems. 
The equations also suggest that in iteration $t$ a $\delta$-error is made twice. This can  be seen by looking at the  function values evaluated at two consecutive points of the sequences $\{\x_\ell\}_{t \ge 0}$ and $\{\p_\ell\}_{t \ge 0}$ generated via the $\delta$-inexact solutions of the auxiliary optimization problems:
\begin{align*}
\Phi(\x_{\ell+1},\p_{\ell+1}) &= f(\x_{\ell+1}) + \la G_1(\x_{\ell+1}),G_2(\p_{\ell+1}) \ra + h(\p_{\ell+1}) \\
&\le f(\x_{\ell}) +\delta_2 + \la G_1(\x_{\ell}),G_2(\p_{\ell+1}) \ra +h(\p_{\ell+1})  \\
& \le f(\x_{\ell})+ \delta_2 + \la G_1(\x_{\ell}),G_2(\p_{\ell}) \ra +  h(\p_{\ell})+ \delta_1  \\
&\le \Phi(\x_{\ell},\p_{\ell}) + 2\cdot \max\{\delta_1, \delta_2\}.
\end{align*}
Next, we estimate the distances between $u^{\delta_1}(x)$ and $u(x)$ as well as between $v^{\delta_2}(p)$ and $v(p)$.
\begin{lem}\label{lem:ux}
	For any $x \in Q_1$ it holds:
	\[
	\|u^{\delta_1}(x) - u(x)\| \le \sqrt{\frac{2\delta_1}{\sigma_2}},
	\] 
	and for any $p \in Q_2$ it holds:
	\[
	\|v^{\delta_2}(p) - v(p)\| \le \sqrt{\frac{2\delta_2}{\sigma_1}}.
	\]
\end{lem}
\begin{proof}
	We apply \eqref{eq:oc} to derive:
	\[
	  \begin{array}{rcl}
	       	h(u^{\delta_1}(x)) + \la G_1(x), G_2(u^{\delta_1}(x)) \ra 
	&\ge& h(u(x)) + \la G_1(x), G_2(u(x)) \ra \\ \\ &&\displaystyle +  \frac{\sigma_2}{2}\|u^{\delta_1}(x) - u(x)\|^2.
	  \end{array}
	\]
	Due to \eqref{eq:genau} we additionally have:
	\[
	 h(u(x)) + \la G_1(x), G_2(u(x)) \ra \ge h(u^{\delta_1}(x)) + \la G_1(x), G_2(u^{\delta_1}(x)) \ra - \delta_1,
	\]
	Altogether, we obtain:
	\[
	\|u^{\delta_1}(x) - u(x)\|^2 \le \frac{2\delta_1}{\sigma_2}.
	\]
	The proof for $ \|v^{\delta_2}(p) - v(p)\|$ follows analogously.  
\end{proof}


Let us elaborate on the continuity properties for operators $u^{\delta_1}(\cdot)$ and $v^{\delta_2}(\cdot)$.

\begin{lem}\label{lem:dist.inex}
	For any $x_1, x_2 \in Q_1$ it holds:
	\[
	\|u^{\delta_1}(x_1) - u^{\delta_1}(x_2)\| \le 2\cdot \sqrt{\frac{2\delta_1}{\sigma_2}} + \frac{L_1\cdot L_2}{\sigma_2}\cdot \|x_1 - x_2\|,
	\]
	and for any $p_1, p_2 \in Q_2$ it holds: 
	\[
	\|v^{\delta_2}(p_1) - v^{\delta_2}(p_2)\| \le 2\cdot \sqrt{\frac{2\delta_2}{\sigma_1}} + \frac{L_1\cdot L_2}{\sigma_1}\cdot \|p_1 - p_2\|.
	\]
\end{lem}

\begin{proof}
	Applying Lemma \ref{lem:ux} and the triangle inequality twice yields:
	\begin{align*}
	\|u^\delta(x_1) - u^\delta(x_2)\| &\le \|u^\delta(x_1) - u(x_1)\| + \|u(x_1) - u^{\delta_1}(x_2)\| \\
	&\le \sqrt{\frac{2\delta_1}{\sigma_2}} + \|u(x_1) - u(x_2)\| + \|u(x_2) - u^{\delta_1}(x_2)\| \\
	&\le \sqrt{\frac{2\delta_1}{\sigma_2}} + \frac{L_1\cdot L_2}{\sigma_2}\cdot \|x_1 - x_2\| + \sqrt{\frac{2\delta_1}{\sigma_2}}.
	\end{align*}
	The last inequality is due to \cite{nesterov2020soft}, where it is shown for any $x_1, x_2 \in Q_1$: 
\[
\|u(x_1)-u(x_2)\| \le \frac{L_1\cdot L_2}{\sigma_2} \|x_1-x_2\|.
\]
Similar reflections yield the result for $\|v^{\delta_2}(p_1) - v^{\delta_2}(p_2)\|$. 
\end{proof}

Let us introduce  inexact versions of the operators $T$ and $S$:
\[
T^\delta(x) = v^{\delta_2}(u^{\delta_1}(x)), \quad S^\delta(p) = u^{\delta_1}(v^{\delta_2}(p)),
\]
which we use to rewrite the update of the  inexact alternating  minimization as 
\begin{equation}\label{eq:upd.inex.oper.T}
\x_{\ell+1} = T^\delta(\x_\ell), \quad 
\p_{\ell+1} = S^\delta(\p_\ell).
\end{equation} 

The following result provides uniform continuity up to an error of the operators defined in \eqref{eq:upd.inex.oper.T}.

\begin{prop}\label{cor:diff.inex.iter}
	For any $\x_1, \x_2 \in Q_1$  it holds:
	\[
	\|T^\delta(\x_1) - T^\delta(\x_2)\| \le \lambda\|\x_1 - \x_2\|+2\cdot \sqrt{\frac{2\delta_2}{\sigma_1}} + 2\cdot \sqrt{\frac{2\delta_1}{\sigma_2}}\cdot \frac{L_1\cdot L_2}{\sigma_1},
	\]
	and for any $\p_1, \p_2 \in Q_2$ it holds:
	\[
	\|S^\delta(\p_1) - S^\delta(\p_2)\| \le + \lambda\|\p_1 - \p_2\|+2\cdot \sqrt{\frac{2\delta_1}{\sigma_2}} + 2\cdot \sqrt{\frac{2\delta_2}{\sigma_1}}\cdot \frac{L_1\cdot L_2}{\sigma_2}.
	\]
\end{prop} 
\begin{proof}
	We apply Lemma \ref{lem:dist.inex} to derive
	\begin{align*}
	\| T^\delta(\x_1) -  T^\delta(\x_2)\|&= \|v^{\delta_2}(u^{\delta_1}(\x_1)) - v^{\delta_2}(u^{\delta_1}(\x_2))\| \\  
	&\le  2\cdot \sqrt{\frac{2\delta_2}{\sigma_1}} + \frac{L_1\cdot L_2}{\sigma_1}\cdot \|u^{\delta_1}(\x_1)) - (u^{\delta_1}(\x_2))\| \\
	&\le 2\cdot \sqrt{\frac{2\delta_2}{\sigma_1}} +  \frac{L_1\cdot L_2}{\sigma_1}\cdot 2\sqrt{\frac{2\delta_1}{\sigma_2}} + \underbrace{\frac{L_1^2\cdot L_2^2}{\sigma_1 \cdot \sigma_2}}_{= \lambda}\cdot \|\x_1 - \x_2\|.
	\end{align*} 
	Again, the second assertion follows similarly.
\end{proof}

Since we cannot rely on the contraction property of $T^\delta$ and $S^\delta$, the convergence analysis of the sequences $\{\x_\ell\}_{\ell\ge 0}$ and $\{\p_\ell\}_{\ell\ge 1}$ becomes involved.  For that, we start with the following auxiliary result.

\begin{lem}\label{lem:T.delta.bound}
		For any $x \in Q_1$ it holds:
		\[
		\left\|T(x)-T^\delta(x)\right\| \le  \sqrt{\frac{2\delta_2}{\sigma_1}} + \frac{L_1\cdot L_2}{\sigma_1}\cdot \sqrt{\frac{2\delta_1}{\sigma_2}},
		\]
and for any $p \in Q_2$ it holds:
		\[
		\left\|S(p)-S^\delta(p)\right\| \le  \sqrt{\frac{2\delta_1}{\sigma_2}}  + \frac{L_1\cdot L_2}{\sigma_2}\cdot \sqrt{\frac{2\delta_2}{\sigma_1}}.
		\]
\end{lem}
\begin{proof}
 We show the first part. It follows by means of Lemmata \ref{lem:ux} and \ref{lem:dist.inex}.
	\begin{align*}
	\left\|T(x)-T^\delta(x)\right\| &= \left\| v(u(x))- v^{\delta_2}(u^{\delta_1}(x))\right\| \\
	&\le 	\left\|v(u(x))- v(u^{\delta_1}(x))\right\| + 	\left\|v(u^{\delta_1}(x))- v^{\delta_2}(u^{\delta_1}(x))\right\|  \\
	&\le \frac{L_1\cdot L_2}{\sigma_1} \cdot\left\|u(x) - u^{\delta_1}(x) \right\| +    \sqrt{\frac{2\delta_2}{\sigma_1}}   \\
	&\le \frac{L_1\cdot L_2}{\sigma_1} \cdot  \sqrt{\frac{2\delta_1}{\sigma_2}} +  \sqrt{\frac{2\delta_2}{\sigma_1}}.
	\end{align*}
	Clearly, the proof of the second part is similar.
\end{proof}

Now we are ready to state the main result concerning  convergence of the inexact alternating minimization scheme. 
\begin{thm}\label{thm:convergence.inexact.alt.min}
	For the inexact alternating minimization scheme holds:
		\begin{equation}\label{eq:optim.inex.estim.x}
		\left\|\x_{\ell+1} - x^*\right\| \le \lambda^{\ell+1} \left\|x_0 - x^*\right\| + \sqrt{\frac{2\delta_2}{\sigma_1}} + \frac{L_1\cdot L_2}{\sigma_1}\cdot \sqrt{\frac{2\delta_1}{\sigma_2}}
		\end{equation}
		and 
		\begin{equation}\label{eq:optim.inex.estim.p}
		\left\|\p_{\ell+1} - p^*\right\| \le \lambda^{\ell}\left[\left\|p_1 - p^*\right\| + \sqrt{\frac{2\delta_1}{\sigma_2}} \right] + \sqrt{\frac{2\delta_1}{\sigma_2}}  + \frac{L_1\cdot L_2}{\sigma_2}\cdot \sqrt{\frac{2\delta_2}{\sigma_1}}.
		\end{equation}
\end{thm}
\begin{proof}
	We apply Lemma \ref{lem:T.delta.bound} to derive:
	\begin{align*}
	\left\|\x_{\ell+1} - x_{\ell+1}\right\| &\le 	\left\|T(x_{\ell}) - T(\x_{\ell})\right\| + 	\left\|T(\x_{\ell}) - T^\delta(\x_{\ell})\right\|\\
	&\le  \lambda\cdot	\left\|x_{\ell} - \x_{\ell}\right\|  + \sqrt{\frac{2\delta_2}{\sigma_1}} + \frac{L_1\cdot L_2}{\sigma_1}\cdot \sqrt{\frac{2\delta_1}{\sigma_2}} \\
	&\le \ldots \le  \lambda^{\ell+1} \underbrace{\left\|x_0-x_0\right\|}_{=0} +  \sqrt{\frac{2\delta_2}{\sigma_1}} + \frac{L_1\cdot L_2}{\sigma_1}\cdot \sqrt{\frac{2\delta_1}{\sigma_2}}.
	\end{align*}
	We are therefore able to estimate the distance of the $\left(t+1\right)$-th iterate to the minimizer:
	\begin{align*}
	\left\|\x_{\ell+1} - x^*\right\| &\le 	\left\|\x_{\ell+1} - x_{\ell+1}\right\| + 	\left\|x_{\ell+1} - x^*\right\| \\
	&\le  \sqrt{\frac{2\delta_2}{\sigma_1}} + \frac{L_1\cdot L_2}{\sigma_1}\cdot \sqrt{\frac{2\delta_1}{\sigma_2}} + \lambda^{\ell+1} \|x_0 - x^*\|.
	\end{align*}
	For the proof of the inequality (\ref{eq:optim.inex.estim.p}) note that the first iterate of the algorithm is not chosen freely. Instead it is the solution of the corresponding optimization problem. Hence, the first iterates of the exact and inexact version are in general not equal, i.e. $p_1 \neq \p_1$, which provides 
	\[
	\left\|\p_{\ell+1} - p_{\ell+1}\right\| \le \lambda^\ell \left\|\p_{1} - p_{1}\right\| + \sqrt{\frac{2\delta_1}{\sigma_2}}  + \frac{L_1\cdot L_2}{\sigma_2}\cdot \sqrt{\frac{2\delta_2}{\sigma_1}}.
	\]
	It remains to recall that $p_1=u(x_0)$ and $\p_1=u^{\delta_1}(x_0)$ and apply Lemma \ref{lem:ux}. The result   \eqref{eq:optim.inex.estim.p} follows then in the same manner as for (\ref{eq:optim.inex.estim.x}).
\end{proof}
According to Theorem \ref{thm:convergence.inexact.alt.min}  the inexact alternating minimization does not converge in general.  Yet, the distance to the  minimizer is bounded by the resp. second term of  the right hand side of inequalities \eqref{eq:optim.inex.estim.x} and \eqref{eq:optim.inex.estim.p}. By taking the limits, we obtain:
\[
\lim_{t \to \infty}  \left\|\x_{\ell+1} - x^*\right\| \le  \sqrt{\frac{2\delta_2}{\sigma_1}} + \frac{L_1\cdot L_2}{\sigma_1}\cdot \sqrt{\frac{2\delta_1}{\sigma_2}} 
\] 
and 
\[
\lim_{t \to \infty}  \left\|\p_{\ell+1} - p^*\right\| \le \sqrt{\frac{2\delta_1}{\sigma_2}}  + \frac{L_1\cdot L_2}{\sigma_2}\cdot \sqrt{\frac{2\delta_2}{\sigma_1}}.
\]
Obviously, convergence is guaranteed if the subproblems can be solved exactly, i.e. if $\delta_1=\delta_2=0$. This is not surprising as in this case the inexact alternating minimization scheme coincides the exact method proposed by \cite{nesterov2020soft}. Therefore, the iterates generated by the inexact alternating minimization scheme \eqref{eq:upd.inex.oper.T} coincide withe the those generated by the exact method. 
Inequalities \eqref{eq:optim.inex.estim.x} and \eqref{eq:optim.inex.estim.p} show that the total error of the inexact alternating minimization scheme can be controlled. Furthermore, large convexity parameters not only improve the rate of convergence for the exact version of the algorithm, but also decrease the total accumulated error in the inexact scenario.

\section{Convergence Analysis} 
\label{sec:Converg}

We analyze the convergence of our network manipulation algorithm by applying the general theory of inexact alternating minimization from Section \ref{sec:inex.alt}. First, we estimate the convexity parameter of 
\[
h(P) = \sum_{i=1}^{N} E_i^*(p_i)
\]
w.r.t. the norm
\[
	\|P\|_\mathbb{H} = \left(\sum_{i=1}^{N}\|p_i\|_1^2\right)^\frac{1}{2}, \quad P \in \Delta_K^N.
\]
It turns out that the strong convexity of $E_i^*$ holds due to Assumption \ref{ass:agents_behavior}. This has been recently shown in \cite{prox}.

\begin{lem}[\cite{prox}]
\label{lem:sc}
   Let the differences $\epsilon^{(k)}_i - \epsilon^{(m)}_i$  of random errors have modes $\bar z^{k,m}_i \in \R$, $k \not = m$.
   Then, the corresponding convex conjugate $E^*_i$ is $\beta_i$-strongly convex w.r.t. the norm $\|\cdot\|_1$, where the  convexity parameter is given by
\[
  \beta_i = \frac{1}{\displaystyle 2\sum_{k=1}^{K} \sum_{m\not =k} g^{k,m}_i \left(\bar z_{k,m}\right)},
\]
and $g^{k,m}_i$ denotes the density function of $\epsilon^{(k)}_i - \epsilon^{(m)}_i$.
\end{lem}

Let us review important discrete choice models in accordance with Assumption \ref{ass:agents_behavior}, where convexity parameters can be explicitly estimated.
\begin{remark}\label{rem:D.C:models}
	 In the multinomial logit model (MNL),  the error terms are IID Gumbel distributed with zero location parameter and variance $\frac{\pi\cdot \mu }{\sqrt{6}}$, where $ \mu > 0$ \cite{palma}. The choice probabilities are:
	\[
	\Prob \left( u^{(k)} + \epsilon^{(k)} = \max_{1 \leq m \leq K} u^{(m)} + \epsilon^{(m)}\right)= \frac{e^{\nicefrac{u^{(k)}}{\mu}}}{\displaystyle
		\sum_{m=1}^{K}e^{\nicefrac{u^{(m)}}{\mu}}}, \quad k=1, \ldots, K.
	\]
	From the choice probabilities we can  conclude that the parameter $\mu$ reflects the randomness of the decision. If the latter converges to zero, this would lead to the deterministic decision based on the observable utility. On the  other hand, very large values of the parameter provide very random choices, tending towards the uniform distribution in the limit. The convex conjugate of the corresponding surplus function is up to an additive constant: 
	\[
	E^*(p)=\mu\sum_{k=1}^{K}  p^{(k)} \cdot\ln p^{(k)}.
	\]
	It is well known from Pinsker inequality that this function is $\mu$-strongly convex w.r.t. the norm $\|\cdot\|_1$.
	Another famous example is the nested logit model (NL) introduced in \cite{gev}. Compared to the MNL, the NL is more appropriate in situations where some of the alternatives are correlated, i.e. the axiom of irrelevance of independent alternatives is violated, see e.\,g. \cite{palma}. In the NL, each alternative $k$ belongs to one of $L$ different nests $N_\ell \subset \{1, \ldots, K\}$ for $\ell= 1, \ldots, L$. The choice probabilities for $k \in N_\ell$, $\ell \in L$ are
	\[
	\Prob \left( u^{(k)} + \epsilon^{(k)} = \max_{1 \leq m \leq K} u^{(m)} + \epsilon^{(j)}\right)= 
	\frac{e^{\mu_\ell \ln \sum_{m \in N_\ell} e^{\nicefrac{u^{(m)}}{\mu_\ell}}}}{\displaystyle
		\sum_{\ell \in L} e^{\mu_\ell \ln \sum_{m \in N_\ell} e^{\nicefrac{u^{(m)}}{\mu_\ell}}}}\cdot   \frac{e^{\nicefrac{u^{(k)}}{\mu_\ell}}}{\displaystyle
		\sum_{m\in N_\ell} e^{\nicefrac{u^{(m)}}{\mu_\ell}}},
	\]
	where the following condition shall be satisfied:
	\[
	0 < \mu_\ell \le 1, \quad \ell = 1, \ldots, L.
	\]
	The parameter $\mu_\ell$ determines the randomness of choices within the $\ell$-th nest. Further, the correlation of alternatives within the same $\ell$-th nest is given by $1 - \mu_\ell^2$.
	The convex conjugate of the NL surplus function has been derived up to an additive constant in \cite{shum}:
\[
   E^*(p) = \sum_{\ell \in L} \mu_\ell \sum_{i\in N_\ell} p^{(m)} \ln p^{(m)} 
   + \sum_{\ell \in L} \left(1-\mu_\ell\right) \left(\sum_{m\in N_\ell} p^{(m)} \right) \ln \left(\sum_{m\in N_\ell} p^{(m)} \right).
\]
	It is $\displaystyle \left(\min_{\ell \in L} \mu_\ell \right)$-strongly convex w.r.t. the norm $\|\cdot\|_1$. In \cite{dynamic}, analogous results were obtained for general nested logit models (GNL). \qed
\end{remark}

We state Lemma \ref{lem:StrCvx.h} concerning the strong convexity of the function $h$ defined in \eqref{eq:dfn.f}.
\begin{lem}\label{lem:StrCvx.h}
	Let the functions $E_i^*$ be $\beta_i$-strongly convex w.r.t. the norm $\|\cdot\|_1$, $i=1, \ldots, N$. Then, the function $h$ is $\sigma_2$-strongly convex w.r.t. the norm $\|\cdot\|_\mathbb{H}$, where 
	\[
	\sigma_2=\min_{1 \le i \le N} \beta_i.
	\]  
\end{lem}
\begin{proof}
	Take any $P,Q \in \Delta_K^N, \alpha \in \left[0,1\right]$. Then the following holds
	\begin{align*}
	&h(\alpha\cdot P + (1-\alpha)\cdot Q) \\ = &\sum_{i=1}^{N}E^*_i\left(\alpha\cdot p_i + (1-\alpha)\cdot q_i\right) \\ \le & \alpha\cdot\sum_{i=1}^{N}E^*_i(p_i) + (1-\alpha) \cdot\sum_{i=1}^{N}E^*_i(q_i)    - \alpha\cdot (1-\alpha) \cdot \sum_{i=1}^{N}\frac{\beta_i}{2} \|p_i-q_i\|_1^2 \\
	\le &\alpha\cdot\sum_{i=1}^{N}E^*(p_i) + (1-\alpha) \cdot\sum_{i=1}^{N}E^*(q_i)    - \alpha\cdot (1-\alpha) \cdot \frac{\sigma_2}{2} \cdot \sum_{i=1}^{N} \|p_i-q_i\|_1^2 \\
	= &\alpha\cdot h(P) + (1-\alpha)\cdot h(Q) - \alpha\cdot (1-\alpha) \cdot \frac{\sigma_2}{2} \cdot \|P-Q\|^2_\mathbb{H}.
	\end{align*}
\end{proof}
Hence, the worst convexity parameter amongst all agents determines the strong convexity of the function $h(P)$.

In order to apply results from Section \ref{sec:inex.alt}, we secondly need to show that 
\[
 f(X) =  -\sum_{k=1}^{K} \frac{1}{\eta_k} \cdot \pi_k\left(M^t\cdot x_k\right)
\]
is strongly convex w.r.t. the norm
\[
\|X\|_{\mathbb{F}} = \left[\sum_{k=1}^{K} \|x_k\|_2^2\right]^\frac{1}{2}, \quad X \in \Delta_n^K.
\]
For that, we need to assume that the underlying network is regular.

\begin{ass}
    Let for the smallest singular value of $M$ hold $\sigma_{\min}\left(M\right) > 0$.
\end{ass}

As a consequence, we are able to estimate the convexity parameter of $f$ w.r.t. the norm $\|\cdot\|_{\mathbb{F}}$.

\begin{lem}\label{lem:StrConv.f}
	The function $f$ is $\sigma_1$-strongly convex w.r.t. the norm $\|\cdot\|_{\mathbb{F}}$, where
	\[
       \sigma_1= \min_{1 \le k \le K} \frac{\tau_k}{\eta_k} \cdot \left[\sigma_{\min}\left(M\right)\right]^{2t}.
	\]
\end{lem}
\proof First, we recall:
\[
 \sigma_{\min}\left(M\right)=\min_{\|x\|_2=1} \|M\cdot x\|_2.   
\]
Hence, we get:
	\begin{align*}
	\sigma_{\min}\left(M^t\right) &= 	\sigma_{\min}\left(M\cdot M^{t-1}\right) = \min_{\|x\|_2=1} \left\| M\cdot M^{t-1}\cdot x \right\|_2 \\
	&\ge \sigma_{\min}\left(M\right)  \cdot\min_{\|x\|_2=1} \left\|M^{t-1}\cdot x\right\|_2 \ge \ldots \ge  \left[\sigma_{\min}\left(M\right)\right]^t.
	\end{align*} 
		 For any $\alpha \in [0,1]$ and $X,Z \in \Delta_n^K$ it holds due to the $\tau_k$-strong convexity of $-\pi_k$ w.r.t. the norm $\|\cdot\|_2$:
		 \begin{align*}
	  & -\pi_k\left(\alpha\cdot M^t\cdot x_k + \left(1-\alpha\right)\cdot M^t\cdot z_k\right) \\  
	  \le & \alpha \cdot \pi_k\left(M^t\cdot x_k\right) - \left(1-\alpha\right)\cdot \pi_k\left(M^t\cdot z_k\right) \\ &-\alpha\cdot \left(1-\alpha\right) \cdot \frac{\tau_k}{2}\cdot  \|M^t x_k - M^t z_k\|_2^2.
	\end{align*} 
Further, we have:
\[
  \|M^t x_k - M^t z_k\|_2 \geq \sigma_{\min}\left(M^t\right)\cdot \| x_k - z_k\|_2\geq  \left[\sigma_{\min}\left(M\right)\right]^t \cdot \| x_k - z_k\|_2.
\]
Hence, the convexity parameter of $-\pi_k\left(M^t x_k\right)$ is $\tau_k \cdot \left[\sigma_{\min}\left(M\right)\right]^{2t}$.
The assertion follows analogously to the proof of Lemma \ref{lem:StrCvx.h}. \qed

Note that the considerations in the proof of Lemma \ref{lem:StrConv.f} also guarantee the existence of a unique minimizer $x^*_k(P)$ for each objective function in \eqref{eq:OrgMin}, i.e. the manipulation matrix $X_*(P)$ is indeed well defined.


It remains to inspect the multiplicative term.  We study the Lipschitz-continuity property of the operator  
\[
   G(X) = \left(g_1(X), \ldots, g_N(X)\right) \in \R^{K \times N},
\]
where
\[
g_i(X) = \left(\|v_i - M^t\cdot x_1\|_2, \ldots, \|v_i - M^t\cdot x_K\|_2\right)^T, \; i=1,\ldots, N. 
\]
For that, the dual norm of $\|\cdot\|_\mathbb{H}$ is required, see \cite{nesterov2020soft}:
\[
\|Z\|^*_\mathbb{H} = \left[\sum_{i=1}^{N} \|z_i\|_\infty^2\right]^{\frac{1}{2}}, \quad Z \in \R^{K \times N}.
\]

\begin{lem}
   The operator $G$ is Lipschitz-continuous with modulus 
   \[
     L_1 = N^\frac{1}{2} \cdot \left[\sigma_{\max}\left(M\right)\right]^t,
   \]
   where $\sigma_{\max}$ denotes the largest singular value of $M$. This is to say that
   \[
\|G(X) - G(Y)\|_{\mathbb{H}}^* \le  L_1 \cdot \|X-Y\|_{\mathbb{F}}, \quad X,Y \in \Delta^K_n.
\]
\end{lem}

\proof The Lipschitz-continuity of $G$ follows mainly from \cite{nesterov2020soft}. In fact, take any $X,Y \in \Delta^K_n$:
\[
\|G(X) - G(Y)\|_{\mathbb{H}}^* = \left[\sum_{i=1}^{N} \|g_i(X) - g_i(Y)\|_\infty^2\right]^{\frac{1}{2}}.
\]
It holds by means of the triangle inequality: 
\[
 \begin{array}{rcl}
   \left| g_i(X)-g_i(Y) \right|& = & \left| \|v_i-M^t\cdot x_k\|_2 -\|v_i-M^t\cdot y_k\|_2 \right| \\ \\ &\le& \|M^t x_k- M^t y_k\|_2 \le \left[\sigma_{\max}\left(M\right)\right]^{t} \cdot \|x_k-y_k\|_2.   
 \end{array}
\]
Therefore,
\begin{align*}
\|g_i(X) - g_i(Y)\|_\infty^2 &= \left(\max_{1 \le k\le K} \left| \|v_i-M^t\cdot x_k\|_2 - \|v_i-M^t\cdot y_k\|_2 \right| \right)^2 \\ &\le  \left[\sigma_{\max}\left(M\right)\right]^{2t} \cdot \max_{1 \le k\le K}\|x_k-y_k\|_2^2 \\ &\le \left[\sigma_{\max}\left(M\right)\right]^{2t} \cdot \sum_{k=1}^{K}\|x_k-y_k\|_2^2,
\end{align*}
and the assertion follows. \qed

Note that all components of $G$ are convex and nonnegative. Moreover, all entries of the matrices $P$ are nonnegative. Due to Lemmata \ref{lem:StrCvx.h} and \ref{lem:StrConv.f}, $\Phi(\cdot,P)$ is strongly convex for any fixed $P \in \Delta_K^N$ and  $\Phi(X,\cdot)$ is strongly convex for any fixed $X \in \Delta_n^K$. We therefore conclude that the alternating update steps of our network manipulation algorithm are well defined. 

Let us finally present our main results on the convergence of the network manipulation algorithm. Recall that the derived constants are as follows:
		\begin{equation}
		    \label{eq:hhh1}
		    \sigma_1=\min_{1 \le k \le K} \frac{\tau_k}{\eta_k} \cdot \left[\sigma_{\min}\left(M\right)\right]^{2t}, \quad \sigma_2=\min_{1 \le i \le N} \beta_i,
		\end{equation}
		 and
		 	\begin{equation}
		    \label{eq:hhh2}
		    L_1 =N^\frac{1}{2} \cdot \left[\sigma_{\max}\left(M\right)\right]^t, \quad L_2=1.
		\end{equation}
Moreover, for the rate of convergence we have:
\	\begin{equation}
		    \label{eq:hhh3}
  \lambda = \frac{L_1^2\cdot L_2^2}{\sigma_1 \cdot \sigma_2} =  \frac{N \cdot \left[\kappa(M)\right]^{2 t}}{\displaystyle \min_{1 \le k \le K} \frac{\tau_k}{\eta_k} \cdot \min_{1 \le i \le N} \beta_i},
\end{equation}
where $\kappa(M)$ denotes the condition number of the matrix $M$. 
In order to establish convergence of the network manipulation algorithm, we need an additional assumption which indicates a certain stability for the model. 
\begin{ass}\label{ass:StrConvCondParam}
	It holds:
	\begin{equation}
	    \label{eq:ass-cr}
	\left[\kappa\left(M\right)\right]^t < {\left(\frac{\displaystyle \min_{1 \le k \le K} \frac{\tau_k}{\eta_k} \cdot \min_{1 \le i \le N} \beta_i}{N}\right)}^{\frac{1}{2}}.
	\end{equation}
\end{ass} 
Assumption \ref{ass:StrConvCondParam} is a version of condition \eqref{eq:conv.cond}, which enforces $\lambda < 1$. The latter guarantees strong convexity of the potential function \eqref{eq:Pot.2}.
The straightforward application of Theorem \ref{thm:convergence.inexact.alt.min} now provides:

\begin{thm}
    	Let $\left(X^*,P^*\right) \in \Delta_n^K \times \Delta_K^N$ be the unique minimizer of the potential function \eqref{eq:Pot.2}. Then, for the sequences $\{\tilde{X}_\ell\}_{\ell \ge 0}$ and $\{\tilde{P}_\ell\}_{\ell \ge 1}$ it  holds: 		\begin{equation}\label{eq:optim.inex.estim.xa}
		\left\|\tilde{X}_{\ell+1} - X^*\right\|_{\mathbb{F}} \le \lambda^{\ell+1} \left\|X_0 - X^*\right\|_{\mathbb{F}} + \sqrt{\frac{2\delta_2}{\sigma_1}} + \frac{L_1\cdot L_2}{\sigma_1}\cdot \sqrt{\frac{2\delta_1}{\sigma_2}}
		\end{equation}
		and 
		\begin{equation}\label{eq:optim.inex.estim.pa}
		\left\|\tilde{P}_{\ell+1} - P^*\right\|_{\mathbb{H}} \le \lambda^{\ell}\left[\left\|P_1 - P^*\right\|_{\mathbb{H}} + \sqrt{\frac{2\delta_1}{\sigma_2}} \right] + \sqrt{\frac{2\delta_1}{\sigma_2}}  + \frac{L_1\cdot L_2}{\sigma_2}\cdot \sqrt{\frac{2\delta_2}{\sigma_1}},
		\end{equation}
		where $\sigma_1$, $\sigma_2$, $L_1$, $L_2$, and $\lambda$ are given in \eqref{eq:hhh1}-\eqref{eq:hhh3}.
\end{thm}

Let us comment on Assumption \ref{ass:StrConvCondParam} by elaborating how the model parameters enter into the inequality \eqref{eq:ass-cr}:
\begin{itemize}
    \item[] {\bf Interaction network.} The network structure plays a key role in \eqref{eq:ass-cr}. This is reflected by the condition number $\kappa\left(M\right)$. Its large values cause instability of manipulation, since small changes regarding the aspired states could lead to big changes of optimal starting distributions. 
    In other words, a more predictive pattern of network transitions   speeds up the convergence. The minimum value of the condition number is attained for permutation matrices. In this case, the network interaction is obviously predictable, i.e. organizations can easily determine how network participants distribute information. For similar reasons, the number of interaction periods $t$ has a negative impact on possibility of manipulation. More periods hamper the influence of the starting distribution on the resulting state. Instead, if time progresses the latter is mainly determined just by the network structure independently of the starting distributions.
    \item[] {\bf Agents.} Clearly, more agents $N$ slow down the rate, as organizations have to pay attention to more aspired states. Moreover, large values of $\beta_i$, $i=1, \ldots, N$,  improve the rate of convergence. In order to interpret this fact, we refer to Remark \ref{rem:D.C:models}. There, it has been shown that $\beta_i$'s, can be viewed as measures that agents are still pretty uncertain about their decisions. Due to the duality of discrete choice and rational inattention, agent prone to errors have high information processing costs. Thus, these agent pay less attention to the observable utility, i.e. if their aspired states were reached. The fact that imperfect behavior of agents could help to faster stabilize economic systems was recently also described in \cite{dynamic}.
    \item[] {\bf Organizations.} The parameters $\eta_{k}$, $k=1, \ldots, K$, reflect to which extent organizations take into account their network payoffs. If $\eta_k$'s are relatively large, the organizations focus mainly on reaching the agents' aspired states. 
    It seems surprising that this would not improve the convergence rate of the network manipulation algorithm, but actually worsen it. However, if organizations do not properly act on the network by maximizing their profits, their manipulation power diminishes, since they lose their credibility -- e.g., their followers may be disappointed by getting biased information and leave them. 
    Organizations thus become worthless for agents in terms of manipulation and, as consequence, the network manipulation algorithm becomes less efficient. Hence, the parameter mirrors a certain credibility of the organizations. Further, the impact of $\tau_k$, $k=1, \ldots, K$, on the convergence rate becomes clear if we interpret the latter as measures of organizations' reluctance to change their starting distributions. Form this point of view, the conservative behavior of organizations towards profit maximization makes the network manipulation more stable.
\end{itemize}

Additionally, it is worth to mention that the parameters $\beta_i$'s, $\eta_k$'s, and $\tau_k$'s, which reflect the behavior of agents and organizations, also affect the error bounds in \eqref{eq:optim.inex.estim.xa} and \eqref{eq:optim.inex.estim.pa}. The corresponding interpretation is similar to that for the convergence rate. Namely, the agents' imperfect behavior and organizations' conservatism in profit maximization reduce the accumulated errors.

\bibliographystyle{amsplain}
\bibliography{Literatur_2}

\end{document}